\documentclass[]{amsart}

\newtheorem{theorem}{Theorem}
\newtheorem{lemma}[theorem]{Lemma}

\newtheorem*{1}{Theorem}

\begin{document}

\title{Calculus of generalized hyperbolic tetrahedron}

\author{Ren Guo}

\address{School of Mathematics, University of Minnesota, Minneapolis, MN, 55455}

\email{guoxx170@math.umn.edu}



\begin{abstract} We calculate the Jacobian matrix of the dihedral angles of a generalized hyperbolic tetrahedron as functions of edge lengths and find the complete set of symmetries of this matrix.
\end{abstract}

\maketitle

\section{Introduction}

\subsection{Tetrahedron}

Motivated by studying the polyhedral geometry of triangulated 3-manifolds, Luo \cite{Lu08} calculated the Jacobian matrix of the dihedral angles of a hyperbolic (Euclidean or spherical) tetrahedron as functions of edge lengths. This Jacobian matrix enjoys many symmetries. Some of the symmetries were discovered by Schl\"afli, Wigner \cite{Wi59}, Taylor-Woodward \cite{TW05}. Luo discovered the complete set of symmetries of the Jacobian matrix. 

Denote by $v_1,v_2,v_3,v_4$ the vertexes of a hyperbolic (Euclidean or spherical) tetrahedron. Let $a_{ij}$ and $x_{ij}$ be the dihedral angles and the edge length at the edge $v_iv_j$. The angle $a_{ij}$ for $i,j\in \{1,2,3,4\}$ is a function of the lengths $x_{12}, x_{13}, x_{14},x_{23},x_{24},x_{34}.$

\begin{1}[Luo] 
$$P^{ij}_{rs} = \frac 1{\sin a_{ij} \sin a_{rs}}\frac{\partial a_{ij}}{\partial x_{rs}} $$ satisfies

\begin{enumerate}

\item \textup{(Schl\"afli)} $P^{ij}_{rs} = P_{ij}^{rs}$.

\item \textup{(Wigner, Taylor-Woodward)} $P^{ij}_{kl} =P^{ik}_{jl} = P^{il}_{jk}$ for $\{i,j,k,l\}=\{1,2,3,4\}$.

\item $P^{ij}_{ik} = -P^{ij}_{kl} \cos a_{jk}$ for $\{i,j,k,l\}=\{1,2,3,4\}$.

\item $P^{ij}_{ij}= P^{ij}_{kl} w_{ij}$ for $\{i,j,k,l\}=\{1,2,3,4\}$, where 
$$w_{ij}=\frac{\cos a_{ij}\cos a_{jk}\cos a_{ki}
+\cos a_{ij}\cos a_{jl}\cos a_{li}
+\cos a_{ik}\cos a_{jl}
+\cos a_{il}\cos a_{jk}}{\sin^2a_{ij}}.$$

\item $P^{ij}_{rs} = P^{i'j'}_{r's'}$ where $\{i,j\} \neq \{r,s\}$
and $\{i,j,i',j'\}=\{r,s,r',s'\}=\{1,2,3,4\}$.
\end{enumerate}
\end{1}

Yakut, Savas and Kader \cite{YSK09} calculated the Jacobian matrix for a hyperbolic or spherical tetrahedron and represented each entry of the matrix in terms of $x_{ij}$. The symmetries of the Jacobian matrix are not obvious in their result. 

\subsection{Generalized hyperbolic tetrahedron}

In this paper we calculate the Jacobian matrix of the dihedral angles of a \textit{generalized hyperbolic tetrahedron} as functions of edge lengths. We find a uniform way to deal with the 15 types of generalized hyperbolic tetrahedra. The complete set of symmetries of the Jacobian matrix are discovered. It is a generalization of Luo's result. Our main theorem is the following. 

Let $a_{ij}$ and $x_{ij}$ be the dihedral angle and the edge length at the edge $e_{ij}$ of a generalized hyperbolic tetrahedron. The angle $a_{ij}$ for $i,j\in \{1,2,3,4\}$ is a function of the lengths $x_{12}, x_{13}, x_{14},x_{23},x_{24},x_{34}.$ Denote by $G$ the \textit{Gram matrix} of the generalized hyperbolic tetrahedron. Let $G_{ij}$ be the matrix obtained by deleting the $i-$th row and $j-$th column of the Gram matrix $G$.

\begin{theorem}\label{thm:main}
The Jacobian matrix is $$\frac{\partial(a_{12},a_{13},a_{14},a_{23},a_{24},a_{34})}{\partial(x_{12},x_{13},x_{14},x_{23},x_{24},x_{34})}=
\sqrt{\frac{\det G_{11}\det G_{22}\det G_{33}\det G_{44}}{-(\det G)^3}}DMD,$$
where 
$$D=diag(\sin a_{12},\sin a_{13},\sin a_{14},\sin a_{23},\sin a_{24},\sin a_{34})$$ 
is a diagonal matrix and
$$ 
M=\left(
\begin{matrix}
w_{12}       & -\cos a_{23} & -\cos a_{24} & -\cos a_{13} & -\cos a_{14} & 1               \\
-\cos a_{23} & w_{13}       & -\cos a_{34} & -\cos a_{12} & 1            & -\cos a_{14}    \\
-\cos a_{24} & -\cos a_{34} & w_{14}       & 1            & -\cos a_{12} & -\cos a_{13}    \\
-\cos a_{13} & -\cos a_{12} & 1            & w_{23}       & -\cos a_{34} & -\cos a_{24}    \\
-\cos a_{14} & 1            & -\cos a_{12} & -\cos a_{34} & w_{24}       & -\cos a_{23}    \\ 
1            & -\cos a_{14} & -\cos a_{13} & -\cos a_{24} & -\cos a_{23} & w_{34}          \\
\end{matrix}
\right),
$$
where $$w_{ij}=\frac{\cos a_{ij}\cos a_{jk}\cos a_{ki}
+\cos a_{ij}\cos a_{jl}\cos a_{li}
+\cos a_{ik}\cos a_{jl}
+\cos a_{il}\cos a_{jk}}{\sin^2a_{ij}}.$$

\end{theorem}

Note that the matrix $DMD$ is the same for the 15 types of generalized hyperbolic tetrahedra. What is different is the factor in front of the matrix $DMD$. This factor depends only on the Gram matrix $G$. Luo's result about the symmetries of $P^{ij}_{rs}$ can be interpreted as the symmetries of the matrix $M$ as follows:
\begin{enumerate}
\item $M$ is a symmetric matrix.

\item Any antidiagonal entry of $M$ is $1.$

\item The $(ij,ik)-$th entry of $M$ is $-\cos a_{jk}.$

\item The $(ij,ij)-$th entry of $M$ is $w_{ij}.$

\item Except the diagonal entries, $M$ is symmetric about the antidiagonal axis. 
\end{enumerate}

Theorem \ref{thm:main} can be considered as a generalization from 2 dimensions to 3 dimensions of the derivative of cosine law of a generalized hyperbolic triangle which is studied systematically in \cite{GL09} Lemma 3.5.

Heard \cite{He05} calculated the Jacobian matrix of a generalized hyperbolic tetrahedron and represented each entry of the matrix in terms of $c_{ij}=(-1)^{i+j}\det G_{ij}.$ The symmetries of the Jacobian matrix are not obvious in his result. 

\subsection{Plan of the paper} In section 2, we recall the definition of a generalized hyperbolic tetrahedron and some properties. In section 3, the derivative of the law of cosine for a link at a vertex of a generalized hyperbolic tetrahedron is summarized. In section 4, the derivative of the law of cosine for a face of a generalized hyperbolic tetrahedron is summarized. In section 5, we calculate the determinant of the Gram matrix $G$. Theorem \ref{thm:main} is proved in section 6.

\section{Definition and properties}

The 4-dimensional Minkowski space is the real vector space $\mathbb{R}^4$ equipped the inner product:
$$\langle X,Y \rangle=x_1y_1+x_2y_2+x_3y_3-x_4y_4,$$
where $X=(x_1,x_2,x_3,x_4), Y=(y_1,y_2,y_3,y_4).$

The 3-dimensional hyperbolic space is identified with the positive sheet of the hyperboloid of two sheets:
$$\mathbb{H}^3=\{X\in \mathbb{R}^4: \langle X,X \rangle=-1,x_4>0 \}.$$ 

The positive half of the light cone is 
$$C^+=\{X\in \mathbb{R}^4: \langle X,X \rangle=0,x_4>0 \}$$ and the positive half of the unit sphere is
$$S^+=\{X\in \mathbb{R}^4: \langle X,X \rangle=1,x_4>0\}.$$ 

For any point $v_i\in \mathbb{H}^3\cup C^+ \cup S^+$, the \textit{type} of the point is the number
$$\varepsilon_i=-\langle v_i,v_i \rangle = \left\{
\begin{array}{ccc}
1, & \text{if} & v_i\in \mathbb{H}^3,\\
0, & \text{if} & v_i\in C^+,\\
-1,& \text{if} & v_i\in S^+.
\end{array}
\right.
$$

For any point $v_i\in \mathbb{H}^3\cup C^+ \cup S^+$, we associate $v_i$ a geometric object in $\mathbb{H}^3$ denoted by $P_{\varepsilon_i}(v_i)$ as follows. 

\begin{enumerate}
\item If $v_i\in \mathbb{H}^3,$ then $P_{\varepsilon_i}(v_i)=P_{1}(v_i)=v_i,$ i.e., the point itself.

\item If $v_i\in C^+,$ then 
$P_{\varepsilon_i}(v_i)=P_{0}(v_i)=\{X\in \mathbb{H}^3: \langle X,u \rangle \geq -\frac12\},$ i.e., 
a horoball in $\mathbb{H}^3.$

\item If $v_i\in S^+,$ then 
$P_{\varepsilon_i}(v_i)=P_{-1}(v_i)=\{X\in \mathbb{H}^3: \langle X,u \rangle \geq 0\},$ i.e., 
a half space of $\mathbb{H}^3.$
\end{enumerate}

Given two points $v_i,v_j\in \mathbb{H}^3\cup C^+ \cup S^+$, if $P_{\varepsilon_i}(v_i)\cap P_{\varepsilon_j}(v_j)=\emptyset,$ there is a geodesic segment in $\mathbb{H}^3$ whose length realizes the distance between $P_{\varepsilon_i}(v_i)$ and $P_{\varepsilon_j}(v_j)$. It is denoted by $e_{ij}.$

Given three points $v_i,v_j,v_k\in \mathbb{H}^3\cup C^+ \cup S^+$, if $P_{\varepsilon_i}(v_i)$, $P_{\varepsilon_j}(v_j)$ and $P_{\varepsilon_k}(v_k)$ are disjoint with each other, draw the three geodesic segments $e_{ij},e_{jk},e_{ki}.$ There is a totally geodesic polygon in $\mathbb{H}^3$ bounded by $e_{ij},e_{jk},e_{ki}$ and the boundary of $P_{\varepsilon_i}(v_i)$, $P_{\varepsilon_j}(v_j)$ and $P_{\varepsilon_k}(v_k)$. This polygon is denoted by $\triangle ijk.$ In fact $\triangle ijk$ is a generalized hyperbolic triangle which is studied in \cite{GL09}. 

Given four points $v_1,v_2,v_3,v_4\in \mathbb{H}^3\cup C^+ \cup S^+$ such that $P_{\varepsilon_1}(v_1)$, $P_{\varepsilon_2}(v_2)$, $P_{\varepsilon_3}(v_3)$ and $P_{\varepsilon_4}(v_4)$ are disjoint with each other, there are four polygons $\triangle 234,$ $\triangle 134,$ $\triangle 124,$ $\triangle 123.$ 

If $P_{\varepsilon_i}(v_i)$ is a horoball, the intersections $\partial P_{\varepsilon_i}(v_i) \cap \triangle ijk,$ $\partial P_{\varepsilon_i}(v_i) \cap \triangle ikl,$ and $\partial P_{\varepsilon_i}(v_i) \cap \triangle ilj$ are three Euclidean line segments which bound a Euclidean triangle, i.e., the \textit{link} at $v_i$, denoted by $LK(v_i).$

If $P_{\varepsilon_i}(v_i)$ is a half space, the intersections $\partial P_{\varepsilon_i}(v_i) \cap \triangle ijk,$ $\partial P_{\varepsilon_i}(v_i) \cap \triangle ikl,$ and $\partial P_{\varepsilon_i}(v_i) \cap \triangle ilj$ are three hyperbolic geodesic segments which bound a hyperbolic triangle, i.e., the \textit{link} at $v_i$, denoted by $LK(v_i).$

If the polygons $\triangle 234, \triangle 134, \triangle 124, \triangle 123$ and the links $LK(v_i)$ with $v_i\in C^+\cup S^+$ for $i\in \{1,2,3,4\}$ bound an object in $\mathbb{H}^3$ with positive volume, this object is a \textit{generalized hyperbolic tetrahedron} which is denoted by $T_{1234}$. Its edges are the geodesic segments $e_{ij}$ for $i,j\in \{1,2,3,4\}$. And its faces are $\triangle 234, \triangle 134, \triangle 124, \triangle 123$. 

If $v_i\in \mathbb{H}^3,$ the link $LK(v_i)$ is a spherical triangle which is the intersection of $T_{1234}$ and a sufficiently small sphere centered at $v_i$. 

A generalized hyperbolic tetrahedron is uniquely determined by the four points $v_1,v_2,v_3,v_4$. According to different types of the points $v_i$, there are 15 types of generalized hyperbolic tetrahedra. 

\section{Link}

A link $LK(v_i)$ is a spherical, Euclidean or hyperbolic triangle if $v_i\in \mathbb{H}^3, C^+$ or $S^+$ respectively. 
If $v_i\in \mathbb{H}^3\cup S^+$, denote by $b^i_{kl},b^i_{lj},b^i_{jk}$ the length of edges of $Lk(v_i)$: $\partial P_{\varepsilon_i}(v_i) \cap \triangle ikl,$ $\partial P_{\varepsilon_i}(v_i) \cap \triangle ilj$ and $\partial P_{\varepsilon_i}(v_i) \cap \triangle ijk$. If $v_i\in C^+,$ let $b^i_{kl},b^i_{lj},b^i_{jk}$ be TWICE of the length of edges of $Lk(v_i)$. Denoted by $a_{ij}$ the dihedral angle between the face $\triangle ijk$ and $\triangle ijl$ for $\{i,j,k,l\}=\{1,2,3,4\}$. The dihedral angles $a_{ij},a_{ik},a_{il}$ become the opposite inner angles of $Lk(v_i)$.

We introduce a function of $b^i_{jk}$ and its derivative as follows
\begin{align}
\rho^i_{jk}&=\int_0^{b^i_{jk}}\cos(\sqrt{\varepsilon_i} s)ds,\\
\rho'^i_{jk}&=\cos(\sqrt{\varepsilon_i} b^i_{jk}),
\end{align}
where $\varepsilon_i$ is the type of $v_i.$

The \textit{amplitude} of the link $LK(v_i)$ is defined as \cite{Fe89}
\begin{align}\label{fml:LA}
A^i=\rho^i_{jk}\rho^i_{jl}\sin a_{ij}
\end{align}
which only depends on the link $Lk(v_i).$

The derivative of the law of cosine of a spherical, Euclidean or hyperbolic triangle is derived in \cite{CL03, Lu06} and has the uniform formula.
\begin{lemma}\label{thm:LC}
  \begin{align*}
  \frac{\partial a_{ij}}{\partial b^i_{kl}}&=\frac{\rho^i_{kl}}{A^i},\\
  \frac{\partial a_{ij}}{\partial b^i_{lj}}&=\frac{\rho^i_{kl}}{A^i}(-\cos a_{il}).
  \end{align*}
\end{lemma}

\section{Face}

Let $x_{ij}$ be the length of the edge $e_{ij}$. We introduce a function of $x_{ij}$ and its derivative as follows:
\begin{align}
\tau_{ij}&=\frac12e^{x_{ij}}-\frac12\varepsilon_i\varepsilon_j e^{-x_{ij}},\label{fml:tau1}\\
\tau'_{ij}&=\frac12e^{x_{ij}}+\frac12\varepsilon_i\varepsilon_j e^{-{x_{ij}}}. \label{fml:tau2}
\end{align}

Each face $\triangle jkl$ of $T_{1234}$ is a generalized hyperbolic triangle. It has the edge lengths $x_{kl},x_{lj},x_{jk}$ and the opposite generalized angles $b^j_{kl}, b^k_{lj},b^l_{jk}.$ 

The \textit{amplitude} of the face $\triangle jkl$ is defined as
\begin{align}\label{fml:FA}
A_{jkl}=\tau_{jk}\tau_{jl}\rho^j_{kl}
\end{align}
which only depends on the face.

The derivative of the law of cosine of a generalized hyperbolic triangle is derived in \cite{GL09} and has the uniform formula.
\begin{lemma}[\cite{GL09} Lemma 3.5]\label{thm:FC}
\begin{align*}
  \frac{\partial b^j_{kl}}{\partial x_{kl}}&=\frac{\tau_{kl}}{A_{jkl}},\\
  \frac{\partial b^j_{kl}}{\partial x_{lj}}&=\frac{\tau_{kl}}{A_{jkl}}(-\rho'^l_{jk}).
  \end{align*}
 \end{lemma}
  
\section{Gram matrix} 

Given four points $v_1,v_2,v_3,v_4\in \mathbb{H}^3\cup C^+ \cup S^+,$ the Gram matrix of $T_{1234}$ determined by $v_1,v_2,v_3,v_4$ is defined as 
$$G=(\langle v_i, v_j \rangle)_{4\times 4}.$$ 
 
\begin{lemma} If $i\neq j,$ then
$\langle v_i, v_j \rangle =-\tau'_{ij}.$
\end{lemma}

\begin{proof}
When $\varepsilon_i\varepsilon_j \neq 0,$ it is well-known. See, for example, \cite{Ra06} pp 62-72. When $\varepsilon_i\varepsilon_j=0$, see \cite{He05} pp 7-9. When $\varepsilon_i=\varepsilon_j=0$, the formula was obtained in \cite{Pe87}. Note that we use a different convention in the definition of a horoball from the convention in \cite{He05}. Due to our convention, we have the following.

If $X\in C^+, Y\in \mathbb{H}^3\cup S^+,$ then $\langle X, Y \rangle=-\frac12 e^d$ where $d$ is the distance between the horoball associated to $X$ and the vertex or the half space associated to $Y$. 

If $X, Y\in C^+,$ then $\langle X, Y \rangle=-\frac12 e^d$ where $d$ is the distance between the two horoballs associated to $X$ and $Y$.
\end{proof}

Therefore the Gram matrix can be written as 
$$
G=\left(
\begin{matrix}
-\varepsilon_1 & -\tau'_{12}    & -\tau'_{13}    & -\tau'_{14} \\
-\tau'_{12}    & -\varepsilon_2 & -\tau'_{23}    & -\tau'_{24} \\
-\tau'_{13}    & -\tau'_{23}    & -\varepsilon_3 & -\tau'_{34} \\
-\tau'_{14}    & -\tau'_{24}    & -\tau'_{34}    &-\varepsilon_4
\end{matrix}
\right).
$$

Before calculating $\det G$, we recall an analogy in 2 dimensions. Recall that $G_{ij}$ is the matrix obtained by deleting the $i-$th row and $j-$th column of the Gram matrix $G$. 

\begin{lemma}[\cite{GL09} Lemma 3.3]\label{thm:2Gram}
$G_{ii}$ is the Gram matrix of the face $\triangle jkl$ and 
$$\sqrt{-\det G_{ii}}=A_{jkl}$$ 
where $A_{jkl}$ is the amplitude of the face $\triangle jkl$ (\ref{fml:FA}).
\end{lemma}

\begin{lemma}\label{thm:3Gram}
$$\sqrt{-\det G}=\tau_{ij}\tau_{ik}\tau_{il}A^i,$$
where $A^i$ is the amplitude of the link $LK(v_i)$ (\ref{fml:LA}).
\end{lemma}
\begin{proof} Case 1. If one of vertex is not in $C^+$, say $\varepsilon_1=\pm 1.$ Then
\begin{align*}
\det G&=\frac1{\varepsilon^3_1}\det\left(
\begin{matrix}
\varepsilon_1              & \tau'_{12}                 & \tau'_{13}                 & \tau'_{14} \\
\varepsilon_1\tau'_{12}    & \varepsilon_1\varepsilon_2 & \varepsilon_1\tau'_{23}    & \varepsilon_1\tau'_{24} \\
\varepsilon_1\tau'_{13}    & \varepsilon_1\tau'_{23}    & \varepsilon_1\varepsilon_3 & \varepsilon_1\tau'_{34} \\
\varepsilon_1\tau'_{14}    & \varepsilon_1\tau'_{24}    & \varepsilon_1\tau'_{34}    & \varepsilon_1\varepsilon_4
\end{matrix}
\right)\\
&=\frac1{\varepsilon^3_1}\det\left(
\begin{matrix}
\varepsilon_1   
& \tau'_{12}                 
& \tau'_{13}                 
& \tau'_{14} \\
0               
& \varepsilon_1\varepsilon_2-(\tau'_{12})^2 
& \varepsilon_1\tau'_{23}-\tau'_{12}\tau'_{13}    
& \varepsilon_1\tau'_{24}-\tau'_{12}\tau'_{14} \\
0               
& \varepsilon_1\tau'_{23}-\tau'_{12}\tau'_{13}    
& \varepsilon_1\varepsilon_3-(\tau'_{13})^2  
& \varepsilon_1\tau'_{34}-\tau'_{13}\tau'_{14} \\
0               
& \varepsilon_1\tau'_{24}-\tau'_{12}\tau'_{14}    
& \varepsilon_1\tau'_{34}-\tau'_{13}\tau'_{14}    
& \varepsilon_1\varepsilon_4-(\tau'_{14})^2 
\end{matrix}
\right)\\
&=\det\left(
\begin{matrix}
  \varepsilon_1\varepsilon_2-(\tau'_{12})^2 
& \varepsilon_1\tau'_{23}-\tau'_{12}\tau'_{13}    
& \varepsilon_1\tau'_{24}-\tau'_{12}\tau'_{14} \\
  \varepsilon_1\tau'_{23}-\tau'_{12}\tau'_{13}    
& \varepsilon_1\varepsilon_3-(\tau'_{13})^2  
& \varepsilon_1\tau'_{34}-\tau'_{13}\tau'_{14} \\       
  \varepsilon_1\tau'_{24}-\tau'_{12}\tau'_{14}    
& \varepsilon_1\tau'_{34}-\tau'_{13}\tau'_{14}    
& \varepsilon_1\varepsilon_4-(\tau'_{14})^2 
\end{matrix}
\right)\\
&\overset{(a)}{=}\det\left(
\begin{matrix}
  -(\tau_{12})^2 
& -\tau_{12}\tau_{13}\rho'^1_{23}    
& -\tau_{12}\tau_{14}\rho'^1_{24} \\
  -\tau_{12}\tau_{13}\rho'^1_{23}  
& -(\tau_{13})^2  
& -\tau_{13}\tau_{14}\rho'^1_{34} \\       
  -\tau_{12}\tau_{14}\rho'^1_{24}    
& -\tau_{13}\tau_{14}\rho'^1_{34}   
& -(\tau_{14})^2 
\end{matrix}
\right)\\
&=(\tau_{12}\tau_{13}\tau_{14})^2
\det\left(
\begin{matrix}
  -1 
& -\rho'^1_{23}    
& -\rho'^1_{24} \\
  -\rho'^1_{23}  
& -1  
& -\rho'^1_{34} \\       
  -\rho'^1_{24}    
& -\rho'^1_{34}   
& -1 
\end{matrix}
\right)\\
&\overset{(b)}{=}(\tau_{12}\tau_{13}\tau_{14})^2(-(A^1)^2).
\end{align*}

In the step $(a)$ we use the fact $\varepsilon_i\varepsilon_j-(\tau'_{ij})^2 =  -(\tau_{ij})^2 $ which is easily verified using the definition (\ref{fml:tau1}) and (\ref{fml:tau2}). We also use the law of cosine for a generalized hyperbolic triangle(\cite{GL09} Lemma 3.1):

$$\rho'^j_{kl}=\frac{-\varepsilon_j\tau'_{kl}+\tau'_{jk}\tau'_{jl}}{\tau_{jk}\tau_{jl}}.$$

In the step $(b)$ we use the law of cosine for a spherical or hyperbolic triangle. For details of calculations, see \cite{Fe89} pp 167-171. 

Case 2. If all vertexes are on $C^+,$ i.e., $\varepsilon_i=0$ for $i=1,2,3,4,$ then, by definition (\ref{fml:tau2}),
\begin{align*}
\det G&=\frac1{16}\det\left(
\begin{matrix}
0              & e^{x_{12}}     & e^{x_{13}}     & e^{x_{14}}  \\
e^{x_{12}}     & 0              & e^{x_{23}}     & e^{x_{24}}  \\
e^{x_{13}}     & e^{x_{23}}     & 0              & e^{x_{34}}  \\
e^{x_{14}}     & e^{x_{24}}     & e^{x_{34}}     & 0
\end{matrix}
\right)\\
&=\frac1{16} (e^{2x_{12}+2x_{34}}+e^{2x_{13}+2x_{24}}+e^{2x_{14}+2x_{23}}\\
&\ \ \ \ \ \ \ \ 
-2e^{x_{12}+x_{34}+x_{13}+x_{24}}
-2e^{x_{13}+x_{24}+x_{14}+x_{23}}
-2e^{x_{14}+x_{23}+x_{12}+x_{34}})\\
&\overset{(c)}{=}\frac1{16}e^{2x_{12}+2x_{13}+2x_{14}}
(\frac{(b^1_{34})^4}{16}+\frac{(b^1_{24})^4}{16}+\frac{(b^1_{23})^4}{16}\\
&\ \ \ \ \ \ \ \ \ \ \ \ \ \ \ \ \ \ \ \ \ \ \ \ \ \ \ \ \ \ \ \ \ \ \ \ \ \  -2\frac{(b^1_{34}b^1_{23})^2}{16}-2\frac{(b^1_{23}b^1_{24})^2}{16}-2\frac{(b^1_{34}b^1_{24})^2}{16})\\
&\overset{(d)}{=}(\frac{e^{x_{12}}}2\frac{e^{x_{13}}}2\frac{e^{x_{14}}}2)^2(-(A^1)^2).
\end{align*}

In the step $(c)$ we use the law of cosine for an ideal hyperbolic triangle (\cite{GL09} Appendix A):
$$\frac{(b^i_{jk})^2}{4}=e^{x_{jk}-x_{ij}-x_{ik}}.$$

In the step $(d)$ we use the law of cosine for a Euclidean triangle. In fact, it is Heron's formula of the area of a Euclidean triangle.
\end{proof}

\section{Jacobian matrix}

First, a generalized hyperbolic tetrahedron is determined by its edge lengths uniquely up to isometries. Therefore the dihedral angle $a_{ij}$ for $i,j\in \{1,2,3,4\}$ is a function of the lengths $x_{12}, x_{13}, x_{14},x_{23},x_{24},x_{34}.$ 

\begin{proof}[Proof of Theorem \ref{thm:main}]

To prove of the theorem, we need to calculate $\frac{\partial a_{ij}}{\partial x_{kl}}, \frac{\partial a_{ij}}{\partial x_{ik}}$ and $\frac{\partial a_{ij}}{\partial x_{ij}}.$

\begin{align*}
\frac{\partial a_{ij}}{\partial x_{kl}}
&=\frac{\partial a_{ij}}{\partial b^i_{kl}}\cdot \frac{\partial b^i_{kl}}{\partial x_{kl}}\\
&\overset{(e)}{=}\frac{\rho^i_{kl}}{A^i}\cdot \frac{\tau_{kl}}{A_{ikl}}\\
&\overset{(f)}{=}\frac{\tau_{kl}}{\tau_{il}\tau_{ik}A^i}\\
&=\frac{\tau_{kl}}{\tau_{il}\tau_{ik}A^i}
\cdot\frac1{\sin a_{ij}\sin a_{kl}} \cdot \sin a_{ij}\sin a_{kl}\\
&\overset{(g)}{=}\frac{\tau_{kl}}{\tau_{il}\tau_{ik}A^i}
\cdot \frac{\rho^j_{ik}\rho^j_{ij}\rho^k_{ij}\rho^k_{jl}}{A^jA^k}\cdot \sin a_{ij}\sin a_{kl}\\
&\overset{(h)}{=}\frac{\tau_{kl}}{\tau_{il}\tau_{ik}A^i}
\cdot \frac{\sqrt{-\det G_{ll}}\sqrt{-\det G_{kk}}\sqrt{-\det G_{jj}}\sqrt{-\det G_{ii}}}
{\tau_{ij}\tau_{jk}\tau_{ij}\tau_{jl}\tau_{ki}\tau_{kl}\tau_{kj}\tau_{kl}A^jA^k}
\cdot \sin a_{ij}\sin a_{kl}\\
&\overset{(i)}{=}\sqrt{\frac{\det G_{ii}\det G_{jj}\det G_{kk}\det G_{ll}}{-(\det G)^3}}\cdot \sin a_{ij}\sin a_{kl}.
\end{align*}

In the step $(e)$, Lemma \ref{thm:LC} and Lemma \ref{thm:FC} are used. 

In the step $(f)$, the definition (\ref{fml:FA}) is used. 

In the step $(g)$, the definition (\ref{fml:LA}) is used. 

In the step $(h)$, the definition (\ref{fml:FA}) and Lemma \ref{thm:2Gram} are used. 

In the step $(i)$, Lemma \ref{thm:3Gram} is used.

\begin{align*}
\frac{\partial a_{ij}}{\partial x_{ik}}
&=\frac{\partial a_{ij}}{\partial b^j_{ik}}\cdot \frac{\partial b^j_{ik}}{\partial x_{ik}}\\
&=\frac{\rho^j_{kl}}{A^j}(-\cos a_{jk}) \cdot \frac{\tau_{ik}}{\sqrt{-\det G_{ll}}}\\
&\overset{(j)}{=}\frac{\rho^j_{kl}}{A^j} \cdot \frac{\tau_{ik}}{\sqrt{-\det G_{ll}}}
\cdot \frac1{\sin a_{ij}\sin a_{ik}} \cdot \sin a_{ij}\sin a_{ik}\cdot(-\cos a_{jk})\\
&=\frac{\sqrt{-\det G_{ii}}}{\tau_{jk}\tau_{jl}A^j}\cdot \frac{\tau_{ik}}{\sqrt{-\det G_{ll}}}
\cdot \frac{\sqrt{-\det G_{ll}}\sqrt{-\det G_{kk}}\sqrt{-\det G_{ll}}\sqrt{-\det G_{jj}}}
{\tau_{ij}\tau_{jk}\tau_{ij}\tau_{jl}\tau_{kj}\tau_{ki}\tau_{ki}\tau_{kl}A^jA^k}\\
& \ \ \ \ \ \ \ \ \ \ \ \ \ \ \ \ \ \ \ \ \ \ \ \ \ \ \ \ \ \ \ \ \ \ \ \ \ \ \ \ 
 \ \ \ \ \ \ \ \ \ \ \ \ \ \ \ \ \ \ \ \ \ \ \ \ \     \cdot \sin a_{ij}\sin a_{ik}\cdot(-\cos a_{jk})\\
&=\sqrt{\frac{\det G_{ii}\det G_{jj}\det G_{kk}\det G_{ll}}{-(\det G)^3}} \cdot \sin a_{ij}\sin a_{ik}\cdot(-\cos a_{jk}).
\end{align*}

\begin{align*}
\frac{\partial a_{ij}}{\partial x_{ij}}
&=\frac{\partial a_{ij}}{\partial b^i_{jk}}\cdot \frac{\partial b^i_{jk}}{\partial x_{ij}}
+\frac{\partial a_{ij}}{\partial b^i_{jl}}\cdot \frac{\partial b^i_{jl}}{\partial x_{ij}}\\
\end{align*}
By the symmetry of $k$ and $l$, we only need to calculate the first term.
\begin{align*}
&\frac{\partial a_{ij}}{\partial b^i_{jk}}\cdot \frac{\partial b^i_{jk}}{\partial x_{ij}}\\
&=\frac{\rho^i_{kl}}{A^i}(-\cos a_{ik}) \cdot \frac{\tau_{jk}}{\sqrt{-\det G_{ll}}}(-\rho'^j_{ki})\\
&\overset{(k)}{=}\frac{\rho^i_{kl}}{A^i}\cdot \frac{\tau_{jk}}{\sqrt{-\det G_{ll}}}\cdot \cos a_{ik}
\cdot \frac{\cos a_{jl}+\cos a_{ij}\cos a_{jk}}{\sin a_{ij}\sin a_{jk}}\\
&=\frac{\rho^i_{kl}}{A^i}\cdot \frac{\tau_{jk}}{\sqrt{-\det G_{ll}}} \cdot \frac1{\sin a_{ij}\sin a_{jk}}
\cdot (\cos a_{ik}\cos a_{jl}+\cos a_{ik}\cos a_{ij}\cos a_{jk})\\
&\overset{(l)}{=}\sqrt{\frac{\det G_{ii}\det G_{jj}\det G_{kk}\det G_{ll}}{-(\det G)^3}}
\cdot (\cos a_{ik}\cos a_{jl}+\cos a_{ik}\cos a_{ij}\cos a_{jk}).
\end{align*}

In the step $(k)$, the law of cosine for a hyperbolic or spherical triangle is used. For a Euclidean triangle, we use the fact:
$$1=\frac{\cos a_{jl}+\cos a_{ij}\cos a_{jk}}{\sin a_{ij}\sin a_{jk}}.$$

In the step $(l)$, compare what we need to compute with the result of the step $(j)$ in the last formula. They are the same if we switch $i$ and $j$. Hence the step $(l)$ holds.

Therefore 
\begin{align*}
&\frac{\partial a_{ij}}{\partial x_{ij}}=\sqrt{\frac{\det G_{ii}\det G_{jj}\det G_{kk}\det G_{ll}}{-(\det G)^3}}\\
&\cdot (\cos a_{ik}\cos a_{jl}+\cos a_{ik}\cos a_{ij}\cos a_{jk}
+\cos a_{il}\cos a_{jk}+\cos a_{il}\cos a_{ij}\cos a_{jl})\\
&=\sqrt{\frac{\det G_{ii}\det G_{jj}\det G_{kk}\det G_{ll}}{-(\det G)^3}} \cdot \sin^2 a_{ij}\cdot w_{ij}.
\end{align*}

\end{proof}

\section*{Acknowledgment} 

The author would like to thank Feng Luo for encouragement and valuable suggestion. He thanks Tian Yang for helpful discussion.

\end{document}